\documentclass[11pt]{article}
\usepackage{amsmath}
\usepackage{amssymb}
\usepackage{amsthm}
\usepackage{cases}
\usepackage{url}

\newtheorem{theorem}{Theorem}[section]

\newtheorem{corollary}[theorem]{Corollary}
\newtheorem{proposition}[theorem]{Proposition}
\newtheorem{remark}[theorem]{Remark}

\newcommand\sa{\smallskipamount}
\newcommand\sPP{\\[\sa]\indent}

\begin{document}
\title{An elementary representation of the higher-order Jacobi-type differential equation}
\author{Clemens Markett}
\date{}
\maketitle

\numberwithin{equation}{section}
\numberwithin{theorem}{section}
\begin{abstract}
We investigate the differential equation for the Jacobi-type polynomials which are orthogonal on the interval $[-1,1]$ with respect to the classical Jacobi measure and an additional point mass at one endpoint. This scale of higher-order equations was introduced by J. and R. Koekoek in 1999 essentially by using special function methods. In this paper, a completely elementary representation of the Jacobi-type differential operator of any even order is given. This enables us to trace the orthogonality relation of the Jacobi-type polynomials back to their differential equation. Moreover, we establish a new factorization of the Jacobi-type operator which gives rise to a recurrence relation with respect to the order of the equation. \\ 
\\
Key words: orthogonal polynomials, higher-order linear differential equations, Jacobi-type equations, Jacobi-type polynomials, factorization.\\
\\
2010 Mathematics Subject Classification: 33C47, 34B30, 34L10
\end{abstract}
\section{Introduction and main result}
\label{intro}
In 1999, J. and R. Koekoek \cite{Koe2} established a new class of higher-order linear differential equations satisfied by the “generalized” Jacobi polynomials  $\{P_{n}^{\alpha, \beta, M, N}(x)\} _{n=0}^{\infty},\;\alpha,\,\beta >-1,\;M,\,N \ge 0$. These function systems were introduced and studied by T. H. Koornwinder \cite{Ko} as the orthogonal polynomials with respect to a linear combination of the Jacobi weight function $w_{\alpha,\beta}$  and one or two delta “functions” at the endpoints of the interval $-1 \le x \le 1$,
\begin{equation}
\begin{aligned}
&w_{\alpha,\beta,M,N}(x)=w_{\alpha,\beta}(x)+M\delta(x+1)+N\delta(x-1),\\ 
&w_{\alpha,\beta}(x)=h_{\alpha,\beta}^{-1}(1-x)^{\alpha}(1+x)^{\beta},\\
&h_{\alpha,\beta}=\int_{-1}^{1}(1-x)^{\alpha}(1+x)^{\beta}dx=2^{\alpha+\beta+1}\Gamma(\alpha+1)\Gamma(\beta+1)/\Gamma(\alpha+\beta+2).
 \label{eq1.1}
\end{aligned}
\end{equation}

In the present paper we investigate the so-called Jacobi-type equation with one additional mass point in the weight function, i. e. with either $M$ or $N$  being positive. In terms of the classical Jacobi polynomials \cite [Sec. 10.8]{HTF2}
\begin{equation}
 P_n^{\alpha,\beta}(x)=\frac{(\alpha+1)_n}{n!}
 {}_2F_1(-n,n+\alpha+\beta+1;\alpha+1;\frac{1-x}{2})
 \label{eq1.2}
 \end{equation}
for $n\in\mathbb{N}_{0}=\lbrace 0, 1, \cdots \rbrace$, the Jacobi-type polynomials are given by, cf. \cite{Koe2},\cite{Ko},
\begin{equation}
P_n^{\alpha,\beta,M,N}(x)=P_n^{\alpha,\beta}(x)+MQ_n^{\alpha,\beta}(x)
+NR_n^{\alpha,\beta}(x),\; n\in\mathbb{N}_{0},\;M\cdot N=0,
 \label{eq1.3}
\end{equation}   
where, provided that $A_{n}^{\alpha,\beta}= (\alpha+2)_{n-1} (\alpha+\beta+2)_{n}/\lbrack 2n!(\beta+1)_{n-1}\rbrack$,
\begin{equation}
 Q_n^{\alpha,\beta}(x)=A_n^{\beta,\alpha}(x+1)P_{n-1}^{\alpha,\beta+2}(x),\; n\in\mathbb{N},\; Q_0^{\alpha,\beta}(x)=0,
 \label{eq1.4}
 \end{equation}	 
\begin{equation}
R_n^{\alpha,\beta}(x)=A_n^{\alpha,\beta}(x-1)P_{n-1}^{\alpha+2,\beta}(x),\; n\in\mathbb{N},\; R_0^{\alpha,\beta}(x)=0.
 \label{eq1.5}
 \end{equation}	  
 By means of the well-known relationship $P_n^{\alpha,\beta}(x)=(-1)^{n}P_n^{\beta,\alpha}(-x)$, it follows that
\begin{equation}
 Q_n^{\alpha,\beta}(x)=(-1)^{n}R_n^{\beta,\alpha}(-x),
 \;P_n^{\alpha,\beta,M,0}(x)=(-1)^{n}P_n^{\beta,\alpha,0,M}(-x),\;n\in\mathbb{N}_{0}, M>0.
 \label{eq1.6}
 \end{equation}	 
 Hence it suffices to treat, for instance, the case $M=0, N>0$ in full detail. The corresponding results for $M>0, N=0$  then follow immediately. 
 
 For $\alpha\in\mathbb{N}_{0}$ and any $\beta>-1,\,N>0$, J. and R. Koekoek \cite{Koe2} found that the Jacobi-type polynomials $\{P_{n}^{\alpha, \beta,0,N}(x)\}_{n=0}^{\infty}$ satisfy a linear differential equation of order $2\alpha+4$ which, for our purpose, is conveniently described in the form   
 \begin{equation}
  N\{L_{2\alpha+4,x}^{\alpha,\beta}-\Lambda_{2\alpha+4,n}^{\alpha,\beta}\}y(x)
 +C_{\alpha,\beta}\{L_{2,x}^{\alpha,\beta}-\Lambda_{2,n}^{\alpha,\beta}\}y(x)=0,\;-1<x<1.
  \label{eq1.7}
  \end{equation}	
 The crucial part of this equation consists in the higher-order differential expression 
 \begin{equation}
  L_{2\alpha+4,x}^{\alpha,\beta}y(x)= \sum_{i=1}^{2\alpha+4}d_i^{\alpha,\beta}(x)D_x^i y(x)
      \label{eq1.8}
   \end{equation}	
with coefficient functions involving, among others, a generalized hypergeometric sum, 
 \begin{equation}
 \begin{aligned}
d_i^{\alpha,\beta}(x)=-(\alpha+2)!(\beta+1)_{\alpha+2}&\sum_{k=\max(0,\,i-\alpha-3)}^{i-1}
\frac{(-2)^{i}(\alpha+3)_{i-1-k}(-\alpha-2)_{i-1-k}}{(\beta+1)_{i-1-k}(i-k)!
(i-1-k)!k!}\cdot\\
&\cdot {}_3F_2
\left(\begin{matrix}-k,\alpha+\beta+3,\alpha+i+2-k\\
\beta+i-k,i+1-k \end{matrix};1\right)
\left(\frac{x-1}{2}\right)^{k+1}.
\label{eq1.9}
\end{aligned}
  \end{equation}
Throughout this paper, $D_x^i\equiv (D_x)^i$ denotes the $i$-fold differentiation with respect to $x$. Notice that the highest coefficient function in the sum $(\ref{eq1.8})$ simplifies to $d_{2\alpha+4}^{\alpha,\beta}(x)=(x^2-1)^{\alpha+2}$, see $(\ref{eq2.4})$ below. Furthermore, the eigenvalue parameters and the coupling constant read
\begin{equation}
  \Lambda_{2\alpha+4,n}^{\alpha,\beta}=(n)_{\alpha+2}(n+\beta)_{\alpha+2},\; \Lambda_{2,n}^{\alpha,\beta}=n(n+\alpha+\beta+1),\;C_{\alpha,\beta}=(\alpha+2)!(\beta+1)_{\alpha+1}
      \label{eq1.10}
  \end{equation}
In particular, when $N$ tends to zero, equaton $(\ref{eq1.7})$ reduces to the classical equation for the Jacobi polynomials $P_n^{\alpha,\beta}(x)$, based on the second-order differential operator
 \begin{equation}
 \begin{aligned}
  L_{2,x}^{\alpha,\beta}y(x)&=\lbrace(x^2-1)D_x^2+\left[\alpha-\beta+(\alpha+\beta+2)x\right]D_x\rbrace y(x)\\
  &=(x-1)^{-\alpha}(x+1)^{-\beta}D_x\lbrack(x-1)^{\alpha+1}(x+1)^{\beta+1}D_xy(x)
  \rbrack.
    \label{eq1.11}
  \end{aligned}
  \end{equation}	

For the lowest parameter value $\alpha=0$, the Jacobi-type equation (\ref{eq1.7}) belongs to the very few fourth-order differential equations with polynomial solutions, which were discovered by H. L. Krall \cite{Kr2} in 1940 and further investigated by A. M. Krall and L. L. Littlejohn \cite{Kr}, \cite{KrL}. To our knowledge, it was also Littlejohn who explicitly determined the sixth- and eighth-order Jacobi-type equations for $\alpha=1$ and $\alpha=2$, both in normal and symmetric form. Later and more generally, Kwon, Littlejohn, and Yoon \cite{KLY} characterized all orthogonal polynomial systems satisfying a finite order differential equation of spectral type and named them BKOPS after Bochner and Krall. Zhedanov \cite{Z} stated some necessary conditions for the polynomials $P_{n}^{\alpha, \beta, M, N}(x)$ to belong to this class. In addition he gave a representation of the Jacobi-type differential operator in a form similar to (\ref{eq1.8}), (\ref{eq1.9}).

Almost simultaneously and quite differently to \cite{Koe2}, Bavinck \cite{Ba} used some operator theoretical arguments to present the Jacobi-type differential operator in the “factorized” form 
\begin{equation}
 L_{2\alpha+4,x}^{\alpha,\beta}y(x)=
\prod_{j=0}^{\alpha+1}\lbrace L_{2,x}^{\alpha,\beta}-\frac{2(\alpha+1)}{x-1}+j(\alpha+\beta+1-j)\rbrace y(x).
\label{eq1.12}
 \end{equation} 
   
The main purpose of this paper is to establish the following elementary representation of the higher-order differential expression in equation (\ref{eq1.7}). 
\begin{theorem}
\label{thm1.1}
For any $\alpha\in\mathbb{N}_0,\;\beta>-1$, and for any sufficiently smooth function $y(x)$, 
 \begin{equation}
 L_{2\alpha+4,x}^{\alpha,\beta}y(x)=\frac{x-1}{(x+1)^\beta}D_x^{\alpha+2}\big   \lbrace(x+1)^{\alpha+\beta+2}D_x^{\alpha+2}\lbrack(x-1)^{\alpha+1}y(x)\rbrack\big\rbrace,\;-1<x<1.
 \label{eq1.13}
 \end{equation}
Consequently, the Jacobi-type polynomials $P_{n}^{\alpha, \beta, 0, N}(x), n\in\mathbb{N}_0,$ arise as  eigensolutions of the equation
\begin{equation}
\begin{aligned}
  N&(x-1)^{\alpha+1}D_x^{\alpha+2}\big\lbrace(x+1)^{\alpha+\beta+2}D_x^{\alpha+2}
  \lbrack(x-1)^{\alpha+1}y_n(x)\rbrack\big\rbrace\\
 &+C_{\alpha,\beta}D_x\lbrack(x-1)^{\alpha+1}(x+1)^{\beta+1}D_xy(x)\rbrack
 =\Lambda_{2\alpha+4,n}^{\alpha,\beta,N}(x-1)^{\alpha}(x+1)^{\beta}y_n(x)
       \label{eq1.14}
   \end{aligned}
  \end{equation}	
 \end{theorem}
 with combined eigenvalue parameter 
 \begin{equation}
  \Lambda_{2\alpha+4,n}^{\alpha,\beta,N}=\lbrack N(n+1)_{\alpha+1}(n+\beta)_{\alpha+1}
  +(2)_{\alpha+1}(\beta+1)_{\alpha+1}\rbrack\; n(n+\alpha+\beta+1).
   \label{eq1.15}
   \end{equation}	
   \linebreak[0]   
   \begin{corollary}
   \label{cor1.2}
   For $\beta\in\mathbb{N}_0$ and any $\alpha>-1,\;M>0$, the Jacobi-type orthogonal polynomials
   \begin{equation*}
   y_n(x)= P_n^{\alpha,\beta,M,0}(x)=P_n^{\alpha,\beta}(x)+MQ_n^{\alpha,\beta}(x), \; n\in\mathbb{N}_{0}, 
       \end{equation*}	
   satisfy the linear differential equation of order $2\beta+4$,
    \begin{equation}
     M\lbrace\widetilde{L}_{2\beta+4,x}^{\beta,\alpha}
     -\Lambda_{2\beta+4,n}^{\beta,\alpha}\rbrace y_n(x)
    +C_{\beta,\alpha}\lbrace L_{2,x}^{\alpha,\beta}-\Lambda_{2,n}^{\alpha,\beta}\rbrace y_n(x)=0,\;-1<x<1,
      \label{eq1.16}
     \end{equation}	   
      \begin{equation}
    \widetilde{L}_{2\beta+4,x}^{\beta,\alpha}y_n(x)=
    \frac{x+1}{(x-1)^\alpha}D_x^{\beta+2}\big\lbrace(x-1)^{\alpha+\beta+2}D_x^{\beta+2}
    \lbrack(x+1)^{\beta+1}y_n(x)\rbrack\big\rbrace.
    \label{eq1.17}
    \end{equation}
      \end{corollary}
 \begin{proof} 
 According to the relations (\ref{eq1.6}), we interchange the parameters $\alpha$  and $\beta$ in equation (\ref{eq1.7}) and substitute $x=-\xi$ to obtain 
  \begin{equation*}
  L_{2\beta+4,-\xi}^{\beta,\alpha}y_n(-\xi)=
  \widetilde{L}_{2\beta+4,\xi}^{\beta,\alpha}y_n(\xi),\;
  L_{2,-\xi}^{\beta,\alpha}y_n(-\xi)=
  L_{2,\xi}^{\alpha,\beta}y_n(\xi),\;     
  \Lambda_{2,n}^{\beta,\alpha}=\Lambda_{2,n}^{\alpha,\beta}.
   \end{equation*}	   
         	 \end{proof} 
   
 The proof of Theorem \ref{thm1.1} is carried out in Section \ref{sec:2} by verifying the equivalence of the two representations (\ref{eq1.13}) and (\ref{eq1.8}), (\ref{eq1.9}). Another, direct proof is postponed to Section  \ref{sec:5}. 
   
 In addition to gaining some deeper insight into the nature of the Jacobi-type equation, the new results reveal a number of nice properties that make them accessible for wider applications. In Section \ref{sec:3} we show that the Jacobi-type differential operator is symmetric with respect to the scalar product associated with the weight function $w_{\alpha,\beta,0,N}(x)$. This may lay foundation to a spectral theoretical treatment of the Jacobi-type equation (\ref{eq1.7}). In particular, it follows that its polynomial solutions are mutually orthogonal in the respective space. Another interesting feature to be discussed in Section \ref{sec:4} is a new factorization of the differential operator $L_{2\alpha+4,x}^{\alpha,\beta}$ into a product of $\alpha+2$  linear second-order differential expressions, which is distinct from Bavinck's (\ref{eq1.12}). This leads to a significant recurrence relation with respect to the order of the differential equation. Finally, in Section \ref{sec:6}, we show how certain Jacobi-type equations are related to the equation for the symmetric ultraspherical-type polynomials  $P_{n}^{\alpha,\alpha, N, N}(x),\,n\in\mathbb{N}_0,\,\alpha\in\mathbb{N}_0,\,N >0$. Originally due to R. Koekoek \cite{Koe1}, this equation was recently stated by the author \cite{Ma} in a form analogous to (\ref{eq1.7}) and (\ref{eq1.13}), i. e.  
 \begin{equation}
  N\{L_{2\alpha+4,x}-\Lambda_{2\alpha+4,n}\}y(x)
 +C_{\alpha}\{L_{2,x}^{\alpha,\alpha}-\Lambda_{2,n}^{\alpha,\alpha}\}y(x)=0,
 \;-1<x<1.
   \label{eq1.18}
  \end{equation}	
The constant linking the two terms is given here by $C_{\alpha}=\frac{1}{2}(\alpha+2)(2\alpha+2)!$, while
 \begin{equation}
 \begin{aligned}
  &L_{2\alpha+4,x}y(x)=(x^2-1)D_x^{2\alpha+4}\lbrack(x^2-1)^{\alpha+1}y(x)\rbrack,\;
  \Lambda_{2\alpha+4,n}=(n-1)_{2\alpha+4}\\
  &L_{2,x}^{\alpha,\alpha}=\lbrace(x^2-1)D_x^2+2(\alpha+1)x\,D_x\rbrace y(x),\;\Lambda_{2,n}^{\alpha,\alpha}=n(n+2\alpha+1).
    \label{eq1.19}
  \end{aligned}
  \end{equation}
  \section{Proof of Theorem 1.1}
\label{sec:2}

Proceeding from the sum (\ref{eq1.8}) defining  $L_{2\alpha+4,x}^{\alpha,\beta}$, we rewrite the coefficient functions  $d_i^{\alpha,\beta}(x)$, $1\le i \le \alpha+2$, as follows. Firstly, a repeated use of Vandermonde's summation formula
 \begin{equation*}
{}_2F_1(-n,p,q;1)=(q-p)_n/(q)_n,\; q>0,\; n\in\mathbb{N}_0,
 \end{equation*}
 yields 
 \begin{equation*}
 \begin{aligned}
  {}_3F_2
  \left(\begin{matrix}-k,b,p\\
  c,q \end{matrix};1\right)
    &=\sum_{m=0}^k 
   \frac{(-k)_m(b)_m}{m!\,(c)_m} \sum_{j=0}^m 
   \frac{(-m)_j(q-p)_j}{j!\,(q)_j} \\
    &=\sum_{j=0}^k (-1)^j 
          \frac{(-k)_j(b)_j(q-p)_j}{j!\,(c)_j(q)_j} 
    \sum_{m=j}^k \frac{(j-k)_{m-j}(b+j)_{m-j}}{(m-j)!\,(c+j)_{m-j}}\\
    &=\frac{1}{(c)_k} \sum_{m=0}^k (-1)^m          
      \frac{(-k)_m(c-b)_m (b)_{k-m}(q-p)_{k-m}} 
        {m!\,(q)_{k-m}}
   \end{aligned}
  \end{equation*}
and thus, choosing $p=\alpha+\beta+3$ and $q=\beta+i-k$,
 \begin{equation*}
 \begin{aligned}
  {}_3F_2&\left(\begin{matrix}-k,\alpha+i+2-k,\alpha+\beta+3\\
  i+1-k,\beta+i-k\end{matrix};1\right)=\frac{1}{(i+1-k)_k}\cdot \\
      &\cdot\sum_{m=\max(0,i-\alpha-3)}^{\min(k,\alpha+1)}          
      (-1)^m\frac{(-k)_m(-\alpha-1)_m(\alpha+i+2-k)_{k-m}(i-k-\alpha-3)_{k-m}} 
        {m!\,(\beta+i-k)_{k-m}}
   \end{aligned}
  \end{equation*}
  Inserting this expression into the right-hand side of (\ref{eq1.9}) and observing that 
  \begin{equation*}
   \begin{aligned}
   (\alpha+3)_{i-1-k}(\alpha+i+2-k)_{k-m}&=(\alpha+3)_{i-1-m}=
    (\alpha+i+1-m)!/(\alpha+2)!\\
   (-\alpha-2)_{i-1-k}(i-k-\alpha-3)_{k-m}&=(-\alpha-2)_{i-1-m}\\
   &=(-1)^{i-1-m}(\alpha+2)!/(\alpha+3-i+m)!\\        
   (\beta+1)_{i-1-k}(\beta+i-k)_{k-m}&=(\beta+1)_{i-1-m},\;(i-k)!(i+1-k)_k=i!
     \end{aligned}
    \end{equation*}
 we arrive at the double sum
  \begin{equation}
   \begin{aligned}
   d_i^{\alpha,\beta}(x)= \Lambda_{2\alpha+4,1}^{\alpha,\beta}\frac{2^i}{i!}
   &\sum_{k=\max(0,\,i-\alpha-3)}^{i-1}
   \frac{1}{(i-1-k)!\,k!} \left(\frac{x-1}{2}\right)^{k+1}\cdot\\
   &\sum_{m=\max(0,\,i-\alpha-3)}^{\min(k,\alpha+1)}
      \frac{(-k)_m(-\alpha-1)_m(\alpha+1+i-m)!} 
              {m!\,(\beta+1)_{i-1-m}(\alpha+3-i+m)!}.
     \label{eq2.1}
    \end{aligned}
    \end{equation}
On the other hand, we expand the new representation (\ref{eq1.13}) of $L_{2\alpha+4,x}^{\alpha,\beta}$ in the form
 \begin{equation*}
    \frac{x-1}{(x+1)^\beta}D_x^{\alpha+2}\big\lbrace(x+1)^{\alpha+\beta+2}D_x^{\alpha+2}\lbrack(x-1)^{\alpha+1}y(x)\rbrack\big\rbrace
    =\sum_{i=1}^{2\alpha+4}e_i^{\alpha,\beta}(x)D_x^i y(x).
        \end{equation*}
So we have to verify that $e_i^{\alpha,\beta}(x)=d_i^{\alpha,\beta}(x)$ for all $1 \le i \le 2\alpha+4$  . To begin with, we see that
 \begin{equation*}
   \begin{aligned}
   &\frac{x-1}{(x+1)^\beta}D_x^{\alpha+2}\big\lbrace(x+1)^{\alpha+\beta+2}D_x^{\alpha+2}\lbrack(x-1)^{\alpha+1}y(x)\rbrack\big\rbrace\\
   &=\frac{x-1}{(x+1)^\beta}D_x^{\alpha+2}\big\lbrace(x+1)^{\alpha+\beta+2}
   \sum_{s=1}^{\alpha+2}\binom{\alpha+2}{s}\frac{(\alpha+1)!}{(s-1)!}
   (x-1)^{s-1}D_x^s y(x)\big\rbrace\\
   &= \sum_{s=1}^{\alpha+2}\binom{\alpha+2}{s}\frac{(\alpha+1)!}{(s-1)!}
   \frac{x-1}{((x+1)^{\beta}}D_x^{\alpha+2}\big\lbrace(x+1)^{\alpha+\beta+2}(x-1)^{s-1} D_x^s y(x)\big\rbrace\\
   &=\sum_{s=1}^{\alpha+2}\binom{\alpha+2}{s}\frac{(\alpha+1)!}{(s-1)!}
   \sum_{r=0}^{\alpha+2}\binom{\alpha+2}{r}
    \frac{x-1}{(x+1)^{\beta}}D_x^{\alpha+2-r}\lbrack(x+1)^{\alpha+\beta+2}(x-1)^{s-1}\rbrack D_x^{s+r} y(x).
       \end{aligned}
    \end{equation*}
 Substituting $s+r=i,\;0\le r\le \alpha+2$, the index $s$ ranges over $i-\alpha-2\le s\le i$. Hence, 
 \begin{equation}
   e_i^{\alpha,\beta}(x)=\sum_{s=\max(1,\,i-\alpha-2)}^{\min(\alpha+2,i)}
   \binom{\alpha+2}{s}\frac{(\alpha+1)!}{(s-1)!}
   \binom{\alpha+2}{i-s} a_s^{\alpha,\beta,i}(x)
   \label{eq2.2}
   \end{equation}
 with
  \begin{equation*}
    \begin{aligned}
    &a_s^{\alpha,\beta,i}(x)=\frac{x-1}{(x+1)^\beta}D_x^{\alpha+2-i+s}
    \lbrack(x+1)^{\alpha+\beta+2}(x-1)^{s-1}\rbrack\\
    &=\sum_{t=\max(0,\alpha+3-i)}^{\alpha+2-i+s}\binom{\alpha+2-i+s}{t}
     \frac{x-1}{(x+1)^{\beta}}
     D_x^t\lbrack (x+1)^{\alpha+\beta+2}\rbrack
      D_x^{\alpha+2-i+s-t}\lbrack (x-1)^{s-1}\rbrack\\
     &=\sum_{t=\max(0,\alpha+3-i)}^{\alpha+2-i+s}\binom{\alpha+2-i+s}{t}
         \frac{(\beta+1)_{\alpha+2}}{(\beta+1)_{\alpha+2-t}}
     \frac{(s-1)!(x+1)^{\alpha+2-t}(x-1)^{i-\alpha-2+t}}{(i-\alpha-3+t)!}.     
        \end{aligned}
     \end{equation*}
Inserting the last sum into (\ref{eq2.2}) and interchanging the order of summation we achieve
 \begin{equation}
    e_i^{\alpha,\beta}(x)=
    \sum_{t=\max(0,\alpha+3-i)}^{\min(2\alpha+4-i,\alpha+2)}
    b_t^{\alpha,\beta,i}(x)           \frac{(\beta+1)_{\alpha+2}}{(\beta+1)_{\alpha+2-t}}
     \frac{(\alpha+1)!(x+1)^{\alpha+2-t}(x-1)^{i-\alpha-2+t}}{(i-\alpha-3+t)!}, 
     \label{eq2.3}
     \end{equation}
where the new inner sum, $b_t^{\alpha,\beta,i}(x)$, simplifies to
  \begin{equation*}
      \begin{aligned}
      &b_t^{\alpha,\beta,i}(x)=\sum_{s=t+i-\alpha-2}^{\min(\alpha+2,i)}
       \binom{\alpha+2}{s}\binom{\alpha+2}{i-s}\binom{\alpha+2-i+s}{t} \\
       &=\sum_{s=0}^{\min(2\alpha+4-i-t,\alpha+2-t)}
      \frac{(\alpha+2)!}{(s+t+i-\alpha-2)!(2\alpha+4-i-t-s)!}
      \frac{(\alpha+2)!}{(\alpha+2-t-s)!\;s!\;t!} \\
      &=\binom{\alpha+2}{t}\binom{\alpha+2}{t+i-\alpha-2}
      \frac{(\alpha+3)_{\alpha+2-t}}{(t+i-\alpha-1)_{\alpha+2-t}}=
     \binom{\alpha+2}{t}\binom{2\alpha+4-t}{i}.
      \end{aligned}
       \end{equation*}
 Moreover we use
  \begin{equation*}
       \begin{aligned}
     (x+1)^{\alpha+2-t}(x-1)^{i-\alpha-2+t}&=2^i
     \sum_{r=0}^{\alpha+2-t}\binom{\alpha+2-t}{r}
       \left(\frac{x-1}{2}\right)^{i-r}\\
       &=2^i \sum_{k=t+i-\alpha-3}^{i-1}\binom{\alpha+2-t}{i-1-k}
        \left(\frac{x-1}{2}\right)^{k+1}
       \end{aligned}
        \end{equation*}
 and interchange the order of summation once more to obtain
 \begin{equation*}
  \begin{aligned}
   &e_i^{\alpha,\beta}(x)=\frac{2^i}{i!}
    \sum_{k=\max(0,\,i-\alpha-3)}^{i-1}
     \frac{(\beta+1)_{\alpha+2}(\alpha+2)!}{(i-1-k)!\;k!} \left(\frac{x-1}{2}\right)^{k+1}\cdot\\
   &\sum_{t=\max(0,\,\alpha+3-i)}^{\min(k-i+\alpha+3,2\alpha+4-i)}
      \frac{k!(2\alpha+4-t)!(\alpha+1)!} 
      {(\beta+1)_{\alpha+2-t}t!(2\alpha+4-i-t)!(i-\alpha-3+t)!(\alpha+3-i-t+k)!}.
      \end{aligned}
      \end{equation*}      
  Again, by another index transformation $t=m-i+\alpha+3$, the inner sum reduces to the same expression as in (\ref{eq2.1}).This concludes the proof of Theorem \ref{thm1.1}.
  
  Notice that for $i=2\alpha+4$ , identity (\ref{eq2.1}) and, even more directly, the equivalent identity (\ref{eq2.3}) reduce to
  \begin{equation}
      d_{2\alpha+4}^{\alpha,\beta}(x)=e_{2\alpha+4}^{\alpha,\beta}(x)=
     (x-1)^{\alpha+2}(x+1)^{\alpha+2}=(x^2-1)^{\alpha+2}.
            \label{eq2.4}
             \end{equation}
  
\section{ The orthogonality relation of the eigensolutions of the Jacobi-type equation}
\label{sec:3}
  
The aim of this section is to show that for different eigenvalues  $\Lambda_{2\alpha+4,n}^{\alpha,\beta,N}$, $n\in\mathbb{N}_0$, the solutions of the Jacobi-type equation (\ref{eq1.7}) are orthogonal with respect to the scalar product
\begin{equation}
(f,g)_{w(\alpha,\beta,0,N)}=\int_{-1}^{1}
f(x)\; g(x)\;w_{\alpha,\beta}(x)dx+Nf(1)g(1)=0,\;f,g \in C[-1,1].
 \label{eq3.1}
  \end{equation}
This, in turn, is a direct consequence of the following fundamental result.
\begin{theorem}
\label{thm3.1}
For $\alpha \in \mathbb{N}_0,\;\beta >-1$, and $N>0$, the combined differential operator in equation (\ref{eq1.7}),
\begin{equation}
L_{2\alpha+4,x}^{\alpha,\beta,N}f=N\;L_{2\alpha+4,x}^{\alpha,\beta}f
+C_{\alpha.\beta}\;L_{2,x}^{\alpha,\beta}f,\;f \in C^{(2\alpha+4)}[-1,1], 
 \label{eq3.2}
  \end{equation}
is symmetric with respect to the scalar product (\ref{eq3.1}) by virtue of 
\begin{equation}
(L_{2\alpha+4,x}^{\alpha,\beta,N}f,g)_{w(\alpha,\beta,0,N)}=
(f,L_{2\alpha+4,x}^{\alpha,\beta,N}g)_{w(\alpha,\beta,0,N)},\;f,g \in C^{(2\alpha+4)}[-1,1].
 \label{eq3.3}
  \end{equation}
  \end{theorem}
  
\begin{proposition}
For any function  $f,g \in C^{(2\alpha+4)}[-1,1]$ we define the two integrals
\begin{equation*}
\begin{aligned}
S^{\alpha}(f,g)&=h_{\alpha,\beta}^{-1}\int_{-1}^1
D_x^{\alpha+2}\lbrack(x-1)^{\alpha+1}f(x)\rbrack
D_x^{\alpha+2}\lbrack(x-1)^{\alpha+1}g(x)\rbrack(x+1)^{\alpha+\beta+2}dx,\\
T(f,g)&=h_{\alpha,\beta}^{-1}\int_{-1}^1
f'(x)g'(x)(1-x)^{\alpha+1}(1+x)^{\beta+1}dx.
\end{aligned}
 \end{equation*}	
Then the differential operator (\ref{eq3.2}) has the properties
 \begin{align*}
(i) \quad &
(L_{2\alpha+4,x}^{\alpha,\beta}f,g)_{w(\alpha,\beta)}=
S^{\alpha}(f,g)-2(\alpha+1)\;C_{\alpha.\beta}f'(1)g(1)\\
(ii) \quad &
(L_{2,x}^{\alpha,\beta}f,g)_{w(\alpha,\beta)}=T(f,g)\\
(iii) \quad &
L_{2\alpha+4,x}^{\alpha,\beta}f(x)\big\vert_{x=1}=0,\;
L_{2,x}^{\alpha,\beta}f(x)\big\vert_{x=1}=2(\alpha+1)f'(1).
\end{align*}   
\label{prop3.1}
\end{proposition}	
 \begin{proof} 
 (i) In view of the representation (\ref{eq1.13}) of $L_{2\alpha+4,x}^{\alpha,\beta}$ it follows by an  $(\alpha+2)$-fold integration by parts that
 \begin{equation*}
 \begin{aligned}
 &h_{\alpha,\beta}(L_{2\alpha+4,x}^{\alpha,\beta}f,g)_{w(\alpha,\beta)}\\
 &=(-1)^{\alpha}\int_{-1}^1
 D_x^{\alpha+2} \big\lbrace (x+1)^{\alpha+\beta+2}D_x^{\alpha+2}
 \lbrack(x-1)^{\alpha+1}f(x)\rbrack\big\rbrace(x-1)^{\alpha+1}g(x)dx\\
 &=\sum_{j=0}^{\alpha+1}(-1)^{\alpha+j}
  D_x^{\alpha+1-j}\big\lbrace (x+1)^{\alpha+\beta+2}
  D_x^{\alpha+2}\lbrack(x-1)^{\alpha+1}f(x)\rbrack\big\rbrace D_x^j\lbrack(x-1)^{\alpha+1}g(x)\rbrack \big\vert_{x=-1}^{x=1}\\
  &\;+\int_{-1}^1
  D_x^{\alpha+2}\lbrack(x-1)^{\alpha+1}f(x)\rbrack
  D_x^{\alpha+2}\lbrack(x-1)^{\alpha+1}g(x)\rbrack(x+1)^{\alpha+\beta+2}dx.
  \end{aligned}
    \end{equation*}
   Here, all terms of the sum vanish up to the last one for $j=\alpha+1$, evaluated at $x=1$. Hence,
 \begin{equation*}
  \begin{aligned}
  &(L_{2\alpha+4,x}^{\alpha,\beta}f,g)_{w(\alpha,\beta)}\\
  &=S^{\alpha}(f,g)-h_{\alpha,\beta}^{-1}(x+1)^{\alpha+\beta+2}
  D_x^{\alpha+2}\lbrack(x-1)^{\alpha+1}f(x)\rbrack
    D_x^{\alpha+1}\lbrack(x-1)^{\alpha+1}g(x)\rbrack\big\vert_{x=1}\\
   &=S^{\alpha}(f,g)-h_{\alpha,\beta}^{-1}2^{\alpha+\beta+2}
   (\alpha+2)!f'(1)(\alpha+1)!\;g(1)\\
    &=S^{\alpha}(f,g)-2(\alpha+1)C_{\alpha,\beta}f'(1)g(1).        
     \end{aligned}
     \end{equation*}
  (ii)	Employing the second representation of $L_{2,x}^{\alpha,\beta}$ in (\ref{eq1.11}) we find, now by a simple integration by parts, that 
  \begin{equation*} 
 (L_{2,x}^{\alpha,\beta}f,g)_{w(\alpha,\beta)}=(-1)^{\alpha}h_{\alpha,\beta}^{-1}
   \int_{-1}^1 D_x\lbrace (x-1)^{\alpha+1}(x+1)^{\beta+1}D_xf(x)\rbrace g(x)dx=T(f,g).
  \end{equation*}	
  (iii) The required values of the two differential expressions at $x=1$  follow by definition (\ref{eq1.13}) and (\ref{eq1.11}), respectively.
   \end{proof}     
  {\itshape Proof of Theorem 3.1} In view of Proposition 3.2 and the symmetry relations $S^{\alpha}(f,g)=S^{\alpha}(g,f)$ and $T(f,g)$ $=T(g,f)$, we obtain the required result (\ref{eq3.3}), i.e.
   \begin{equation*}
    \begin{aligned}
    &(L_{2\alpha+4,x}^{\alpha,\beta,N}f,g)_{w(\alpha,\beta,0,N)}\\
    &=(\lbrace N\;L_{2\alpha+4,x}^{\alpha,\beta}+C_{\alpha,\beta}
    L_{2,x}^{\alpha,\beta}\rbrace f,g)_{w(\alpha,\beta)}
    +N\lbrace N\;L_{2\alpha+4,x}^{\alpha,\beta}+C_{\alpha,\beta}
        L_{2,x}^{\alpha,\beta}\rbrace f(x)\big\vert_{x=1}\;g(1)\\
    &=N\;S^{\alpha}(f,g)-N\;2(\alpha+1)C_{\alpha,\beta}f'(1)g(1)
    +C_{\alpha,\beta}T(f,g)+N\;C_{\alpha,\beta}2(\alpha+1)f'(1)g(1)\\
    &= (f,L_{2\alpha+4,x}^{\alpha,\beta,N}g)_{w(\alpha,\beta,0,N)}.
    \end{aligned}
       \end{equation*}
  \begin{corollary}
     \label{cor3.1}
     For $\alpha\in\mathbb{N}_0,\;\beta>-1,\;N>0$, the polynomial eigensolutions of equation (\ref{eq1.7}),\;$y_n(x)= P_n^{\alpha,\beta,0,N}(x)$, $n\in\mathbb{N}_0$, satisfy the orthogonality relation, for $n \ne m$,
     \begin{equation}
     (y_n,y_m)_{w(\alpha,\beta,0,N)}=\int_{-1}^{1}
     y_n(x)\; y_m(x)\;w_{\alpha,\beta}(x)dx+Ny_n(1)y_m(1)=0.
      \label{eq3.4}
      \end{equation}	
      \end{corollary}     
   \begin{proof} 
   In view of equation (\ref{eq1.14}) and the symmetry property stated in Theorem \ref{thm3.1}, we have 
   \begin{equation*}
    \begin{aligned}
   &\left(\Lambda_{2\alpha+4,n}^{\alpha,\beta,N}-\Lambda_{2\alpha+4,m}^{\alpha,\beta,N} \right) (y_n,y_m)_{w(\alpha,\beta,0,N)}\\
   &=(L_{2\alpha+4,x}^{\alpha,\beta,N}y_n,y_m)_{w(\alpha,\beta,0,N)}-
   (y_n,L_{2\alpha+4,x}^{\alpha,\beta,N}y_m)_{w(\alpha,\beta,0,N)}=0.
    \end{aligned}
   \end{equation*}
   Since the difference of the eigenvalues on the left-hand side do not vanish for $n \ne m$, the assertion follows.
   \end{proof}  
  \section{A new factorization of the Jacobi-type differential equation} 
  \label{sec:4}
  
  Let us begin with a slightly different version of Bavinck's factorization formula (\ref{eq1.12}). Setting $y(x)=(x-1)u(x)$  and recalling the definition (\ref{eq1.11}) we obtain 
  \begin{equation}
   \begin{aligned}
   L_{2\alpha+4,x}^{\alpha,\beta}[(x-1)u(x)] &=\prod_{j=0}^{\alpha+1}\big\lbrace L_{2,x}^{\alpha,\beta}-\frac{2(\alpha+1)}{x-1}+j(\alpha+\beta+1-j)\big\rbrace [(x-1)u(x)]\\
   &=(x-1)\prod_{j=0}^{\alpha+1}\big\lbrace L_{2,x}^{\alpha+2,\beta}+(j+1)(\alpha+\beta+2-j)\big\rbrace u(x).
  \label{eq4.1}
    \end{aligned}
     \end{equation}
  Here, the second identity follows by successively applying, for any $j=0,1,\dots,\alpha+1$, 
   \begin{equation}
     \begin{aligned}
     &\lbrace (x^2-1)D_x^2+[\alpha-\beta+(\alpha+\beta+2)x]D_x  -\frac{2(\alpha+1)}{x-1}+j(\alpha+\beta+1-j)\rbrace [(x-1)u(x)]\\
     &=(x-1)\lbrace (x^2-1)D_x^2+[\alpha-\beta+2+(\alpha+\beta+4)x]\,D_x+  (j+1)(\alpha+\beta+2-j)\rbrace u(x).
    \label{eq4.2}
    \end{aligned}
       \end{equation} 
       \linebreak[1]
  \begin{corollary}
    \label{cor4.1}
    \cite[(2.6)]{Ba} Let the functions $R_n^{\alpha,\beta}(x)$, $n\in\mathbb{N}$, be defined as in (\ref{eq1.5}). Then 
    \begin{equation}
    L_{2\alpha+4,x}^{\alpha,\beta}R_n^{\alpha,\beta}(x)=\Lambda_{2\alpha+4,n}^{\alpha,\beta}R_n^{\alpha,\beta}(x),\;n\in\mathbb{N}.
     \label{eq4.3}
     \end{equation}	
        \end{corollary}  
    \begin{proof} 
    In view of formula (\ref{eq4.1}) and the Jacobi equation with first parameter being increased to $\alpha+2$,
     \begin{equation}
      \begin{aligned}
    &L_{2\alpha+4,x}^{\alpha,\beta}R_n^{\alpha,\beta}(x)\\
    &=A_n^{\alpha,\beta}(x-1)\prod_{j=0}^{\alpha+1}\lbrace L_{2,x}^{\alpha+2,\beta}+(j+1)(\alpha+\beta+2-j)\rbrace P_{n-1}^{\alpha+2,\beta}(x)\\
    &=A_n^{\alpha,\beta}(x-1)\prod_{j=0}^{\alpha+1}\lbrace (n-1)(n+\alpha+\beta+2)+(j+1)(\alpha+\beta+2-j)\rbrace\cdot P_{n-1}^{\alpha+2,\beta}(x)\\
    &= \prod_{j=0}^{\alpha+1}\lbrace (n+j)(\alpha+\beta+1-j)\rbrace \cdot A_n^{\alpha,\beta}(x-1)P_{n-1}^{\alpha+2,\beta}(x)\\ 
    &=(n)_{\alpha+2}(n+\beta)_{\alpha+2}A_n^{\alpha,\beta}(x-1) P_{n-1}^{\alpha+2,\beta}(x)=\Lambda_{2\alpha+4,n}^{\alpha,\beta}
    R_n^{\alpha,\beta}(x).
     \end{aligned}
     \label{eq4.4}
     \end{equation}  
       \end{proof}     
  Another elegant proof is solely based on our representation (\ref{eq1.13}) of $L_{2\alpha+4,x}^{\alpha,\beta}$. In order to carry out the two higher-order differentiations $D_x^{\alpha+2}$ occurring there, we iteratively use the following two differentiation formulas which may be derived from standard properties of the Jacobi polynomials \cite [Sec.10.8]{HTF2},
   \begin{equation}
      D_x \lbrack (x-1)^\gamma P_n^{\gamma,\delta}(x) \rbrack=
      (n+\gamma) (x-1)^{\gamma-1}P_n^{\gamma-1,\delta+1}(x),\;\gamma>0,\;\delta>-1,
   \label{eq4.5}
   \end{equation}
   \begin{equation}
        D_x \lbrack (x+1)^\delta P_n^{\gamma,\delta}(x) \rbrack=
        (n+\delta) (x+1)^{\delta-1}P_n^{\gamma+1,\delta-1}(x),\;\gamma>-1,\;\delta>0.
    \label{eq4.6}
    \end{equation}     	
  Then we get, for any $n\in\mathbb{N}$,
   \begin{equation}
   \begin{aligned}
    L_{2\alpha+4,x}^{\alpha,\beta}&R_n^{\alpha,\beta}(x)=
    \frac{x-1}{(x+1)^\beta}D_x^{\alpha+2}\big\lbrace(x+1)^{\alpha+\beta+2}
    D_x^{\alpha+2}\lbrack(x-1)^{\alpha+2}A_n^{\alpha,\beta}
    P_{n-1}^{\alpha+2,\beta}(x)\rbrack\big\rbrace \\   
    &=A_n^{\alpha,\beta}\frac{x-1}{(x+1)^\beta}D_x^{\alpha+2}
    \big\lbrace(x+1)^{\alpha+\beta+2}(n)_{\alpha+2}       
    P_{n-1}^{0,\alpha+\beta+2}(x)\big\rbrace\\
    &=A_n^{\alpha,\beta}(x-1)(n+\beta)_{\alpha+2}(n)_{\alpha+2}
      P_{n-1}^{\alpha+2,\beta}(x)=
      \Lambda_{2\alpha+4,n}^{\alpha,\beta}R_n^{\alpha,\beta}(x).
      \end{aligned}
      \label{eq4.7}
      \end{equation}  
             
  The main purpose of this section is to present a factorization of the Jacobi-type differential operator which is distinct from (\ref{eq1.2}) and more reminiscent of our non-commutative factorization of the symmetric ultraspherical-type equation \cite[Thm. 4.1]{Ma}.
  \begin{theorem}
  \label{thm4.1}
   For $\alpha \in \mathbb{N}_0,\;\beta >-1$, the Jacobi-type differential operator (\ref{eq1.13}) can be factorized by
  \begin{equation}
  \begin{aligned}
  &L_{2\alpha+4,x}^{\alpha,\beta}y(x)=\prod_{j=0}^{\alpha+1}\big\lbrace L_{2,x}^{2j-1,\beta}-\frac{4j}{x-1}+j(j+\beta)\big\rbrace y(x),\\
  &L_{2,x}^{2j-1,\beta}=(x^2-1)D_x^2+[2j-\beta-1+(2j+\beta+1)x]D_x,\;
    j=0,1,\dots,\alpha+1.
    \end{aligned}
    \label{eq4.8}
    \end{equation}
     \end{theorem}
  Here, the product $\prod_{j=0}^{\alpha+1}$ is understood as a successive application of each second-order operator to the respective function on its right-hand side, in the order from $j=0$ to $j=\alpha+1$.
  \begin{proof} 
  Analogously to (\ref{eq4.1}),(\ref{eq4.2}), identity (\ref{eq4.8}) is equivalent to
   \begin{equation}
   L_{2\alpha+4,x}^{\alpha,\beta}[(x-1)u(x)] =(x-1)\prod_{j=0}^{\alpha+1}\big\lbrace L_{2,x}^{2j+1,\beta}+(j+1)(j+\beta+1)
   \big\rbrace u(x).
    \label{eq4.9}
    \end{equation}
    Moreover, it is not hard to see that 
   \begin{equation}
    \begin{aligned}
    (x+1)^{\beta} &\lbrace L_{2,x}^{2j+1,\beta}+(j+1)(j+\beta+1)\big\rbrace u(x)\\
    &= \big\lbrace L_{2,x}^{2j+1,-\beta}+(j+1)(j-\beta+1)\big\rbrace \lbrack (x+1)^{\beta} u(x) \rbrack.
    \end{aligned}
      \label{eq4.10}
      \end{equation}
  So we have to verify that 
  \begin{equation}
   \begin{aligned}
      (x+1)^{\beta}&(x-1)^{-1} L_{2\alpha+4,x}^{\alpha,\beta}[(x-1)u(x)]\\
      &=D_x^{\alpha+2}\big\lbrace(x+1)^{\alpha+\beta+2}
          D_x^{\alpha+2}\lbrack(x-1)^{\alpha+2}u(x)\rbrack\big\rbrace\\
     &=\prod_{j=0}^{\alpha+1}\big\lbrace L_{2,x}^{2j+1,-\beta}+(j+1)(j-\beta+1)\big\rbrace \lbrack (x+1)^{\beta}u(x)\rbrack.
       \end{aligned}
       \label{eq4.11}
      \end{equation}  
  This, however, follows by successively applying the following recurrence relation. 
      \end{proof}   
 \begin{proposition}
 For any $j\in\mathbb{N}_0$, the expression $u_j^\beta(x)=(x+1)^{j+\beta}D_x^j [(x-1)^j u(x)]$ satisfies
  \begin{equation}
 D_x^{j+1}u_{j+1}^\beta(x)=\big\lbrace L_{2,x}^{2j+1,-\beta}+(j+1)(j-\beta+1)\big\rbrace
 D_x^ju_j^\beta(x).
   \label{eq4.12}
    \end{equation}	
 \end{proposition}	 
  \begin{proof} 
  Since
   \begin{equation*}
  (x+1)^{j+\beta+1}(x-1)D_x^{j+1} [(x-1)^j u(x)]=
  (x^2-1)D_x u_j^\beta(x)-(j+\beta)(x-1)u_j^\beta(x),
        \end{equation*}	 
   we have, recalling $L_{2,x}^{2j+1,-\beta}=(x^2-1)D_x^2+[2j+\beta+1+(2j-\beta+3)x]\,D_x$, that
   \begin{equation*}
   \begin{aligned}
    &D_x^{j+1}u_{j+1}^\beta(x)= D_x^{j+1}\big\lbrace
    (x+1)^{j+\beta+1}D_x^{j+1}[(x-1)(x-1)^j u(x)]\big\rbrace\\
    &=D_x^{j+1}\big\lbrace
    (x+1)^{j+\beta+1}(x-1)D_x^{j+1}[(x-1)^j u(x)]+(j+1)(x+1)u_j^\beta(x)\big\rbrace\\
    &=D_x^{j+1}\big\lbrace (x^2-1)D_x u_j^\beta(x) +[2j+\beta+1+(1-\beta)x] u_j^\beta(x)\big\rbrace\\
    &=(x^2-1)D_x^{j+2}u_j^\beta(x)+(j+1)2xD_x^{j+1}u_j^\beta(x)
    +(j+1)j\;D_x^ju_j^\beta(x)\\
    &\quad +[2j+\beta+1+(1-\beta)x]D_x^{j+1} u_j^\beta(x)
    +(j+1)(1-\beta)D_x^j u_j^\beta(x)\\    
    &=(x^2-1)D_x^{j+2}u_j^\beta(x)
    +[2j+\beta+1+(2j-\beta+3)x]D_x^{j+1} u_j^\beta(x)\\
     &\quad +(j+1)(j-\beta+1)D_x^ju_j^\beta(x)\\ 
    &=\big\lbrace L_{2,x}^{2j+1,-\beta}+(j+1)(j-\beta+1)\big\rbrace D_x^ju_j^\beta(x).   
    \end{aligned}
   \end{equation*} 
        \end{proof}
   \begin{corollary}
     \label{cor4.2}
    For $\alpha\in\mathbb{N}_0,\;\beta>-1$, the Jacobi-type differential operator (\ref{eq1.13}) satisfies the recurrence relation 
     \begin{equation}
     \begin{aligned}
      L_{2\alpha+4,x}^{\alpha,\beta}y(x)=&
      \big\lbrace (x^2-1)D_x^2+[2\alpha-\beta+1+(2\alpha+\beta+3)x]\,D_x-\\
      &-\frac{4(\alpha+1)}{x-1}+(\alpha+1)(\alpha+\beta+1)\big\rbrace L_{2\alpha+2,x}^{\alpha-1,\beta}y(x),
       \end{aligned}
        \label{eq4.13}
       \end{equation}
       where, for $\alpha=0$, the recurrence starts with the operator
        \begin{equation}
        L_{2,x}^{-1,\beta}y(x)=\lbrace (x^2-1)D_x^2+(\beta+1)(x-1)D_x \rbrace y(x).
        \label{eq4.14}
       \end{equation}
       Notice that  the operator (\ref{eq4.14}) is obtained as well, if we formally take $\alpha=-1$ in the definition (\ref{eq1.11}) of the Jacobi differential operator $L_{2,x}^{\alpha,\beta}$. 
   \end{corollary}	
   \section{A direct verification of the Jacobi-type differential equation}
   \label{sec:5}
   
   Now we are in a position to prove, just by utilizing the representation (\ref{eq1.13}) of $ L_{2\alpha+4,x}^{\alpha,\beta}$, that the Jacobi-type polynomials $P_n^{\alpha,\beta,0,N}(x)=P_n^{\alpha,\beta}(x)+NR_n^{\alpha,\beta}(x)$,
   $n\in\mathbb{N}_{0}$, solve equation (\ref{eq1.7}). By Corollary \ref{cor4.1} and the classical Jacobi equation we already know that
   \begin{equation*}
      \lbrace L_{2\alpha+4,x}^{\alpha,\beta}-\Lambda_{2\alpha+4,n}^{\alpha,\beta}\rbrace R_n^{\alpha,\beta}(x)=0,\quad  
     \lbrace L_{2,x}^{\alpha,\beta}-\Lambda_{2,n}^{\alpha,\beta}\rbrace P_n^{\alpha,\beta}(x)=0.
      \end{equation*}  
   So it remains to show that
   \begin{equation}
   \lbrace L_{2\alpha+4,x}^{\alpha,\beta}-\Lambda_{2\alpha+4,n}^{\alpha,\beta}
   \rbrace P_n^{\alpha,\beta}(x)+C_{\alpha,\beta}  
    \lbrace L_{2,x}^{\alpha,\beta}-\Lambda_{2,n}^{\alpha,\beta}\rbrace R_n^{\alpha,\beta}(x)=0,\;-1<x<1.
    \label{eq5.1}
    \end{equation}       
    This is obtained by combining the following two identities, of which the latter one turns out to be crucial here. 
     \begin{theorem}
      \label{thm5.1}
       For all $n \in \mathbb{N}$, there hold
      \begin{equation}
      (i) \quad 
        C_{\alpha,\beta} \big\lbrace L_{2,x}^{\alpha,\beta}-\Lambda_{2,n}^{\alpha,\beta}\big\rbrace R_n^{\alpha,\beta}(x)=\Lambda_{2\alpha+4,n}^{\alpha,\beta}
        \frac {(\alpha+1)(\alpha+2)}{n(n+\alpha+1)}  P_{n-1}^{\alpha+2,\beta}(x),
         \label{eq5.2}
         \end{equation}  
      \begin{equation}
       (ii) \quad 
        L_{2\alpha+4,x}^{\alpha,\beta}P_n^{\alpha,\beta}(x)=
        \Lambda_{2\alpha+4,n}^{\alpha,\beta}\big\lbrace
        P_n^{\alpha,\beta}(x)-\frac {(\alpha+1)(\alpha+2)}{n(n+\alpha+1)}  P_{n-1}^{\alpha+2,\beta}(x)\big\rbrace.
         \label{eq5.3}
       \end{equation}
         \end{theorem}
    \begin{proof} 
    (i) Applying identity (\ref{eq4.2}) for $j=0$ as well as the Jacobi equation with increased parameter $\alpha+2$, we find that
     \begin{equation*}
      \big\lbrace L_{2,x}^{\alpha,\beta}-\frac{2(\alpha+1)}{x-1}-
      \Lambda_{2,n}^{\alpha,\beta}\big\rbrace \lbrack (x-1)P_{n-1}^{\alpha+2,\beta}(x)\rbrack
      =(x-1)\big\lbrace L_{2,x}^{\alpha+2,\beta}-\Lambda_{2,n-1}^{\alpha+2,\beta}\big\rbrace
       P_{n-1}^{\alpha+2,\beta}(x)=0. 
      \end{equation*}  
    So, by definition (\ref{eq1.4}) of $R_n^{\alpha,\beta}(x)$, 
     \begin{equation*}
      \big\lbrace L_{2,x}^{\alpha,\beta}-\Lambda_{2,n}^{\alpha,\beta}\big\rbrace R_n^{\alpha,\beta}(x)=\frac{2(\alpha+1)}{x-1}R_n^{\alpha,\beta}(x)=
      \frac{(\alpha+1)_n(\alpha+\beta+2)_n}{n!\;(\beta+1)_{n-1}}
      P_{n-1}^{\alpha+2,\beta}(x).
      \end{equation*}  
    Identity (\ref{eq5.2}) then follows because of 
     \begin{equation*}
     (\alpha+2)!\;(\beta+1)_{\alpha+1}
     \frac{(\alpha+1)_n(\alpha+\beta+2)_n}{n!\;(\beta+1)_{n-1}}
     =(n)_{\alpha+2}(n+\beta)_{\alpha+2}
          \frac{(\alpha+1)(\alpha+2)}{n(n+\alpha+1)}.
      \end{equation*}  
     (ii) In order to carry out the differentiations in 
     \begin{equation}
      L_{2\alpha+4,x}^{\alpha,\beta}P_n^{\alpha,\beta}(x)=
      \frac{x-1}{(x+1)^\beta}D_x^{\alpha+2}\big\lbrace(x+1)^{\alpha+\beta+2}D_x^{\alpha+2}\lbrack(x-1)^{\alpha+1}P_n^{\alpha,\beta}(x)\rbrack\big\rbrace,
      \label{eq5.4}
       \end{equation}   
      we cannot use the formulas (\ref{eq4.5}), (\ref{eq4.6}) right away as in (\ref{eq4.7}). Instead, we first adjust the parameters of the Jacobi polynomials by employing the well-known identities 
     \begin{equation}
      P_n^{\alpha,\beta}(x)=P_n^{\alpha+1,\beta-1}(x)-P_{n-1}^{\alpha+1,\beta}(x)
      \label{eq5.5}
      \end{equation}   
     \begin{equation}
      (2n+\alpha+\beta+1)P_n^{\alpha,\beta}(x)= (n+\alpha+\beta+1)P_n^{\alpha+1,\beta}(x)-
       (n+\beta)P_{n-1}^{\alpha+1,\beta}(x)
      \label{eq5.6}
      \end{equation}   
      \begin{equation}
       (2n+\alpha+\beta+2)\frac{1-x}{2}P_n^{\alpha+1,\beta}(x)=    (n+\alpha+1)P_n^{\alpha,\beta}(x)-
       (n+1)P_{n+1}^{\alpha,\beta}(x)
        \label{eq5.7}
       \end{equation}   
        \begin{equation}
        D_xP_n^{\gamma,\delta}(x)=\frac{1}{2}(n+\gamma+\delta+1)
        P_{n-1}^{\gamma+1,\delta+1}(x).
        \label{eq5.8}
        \end{equation}
       In fact, using (\ref{eq5.5}), (\ref{eq4.6}), (\ref{eq5.8}), and (\ref{eq5.5}) again, we obtain
       \begin{equation*}
        \begin{aligned}
        D_x^{\alpha+2}&\lbrack(x-1)^{\alpha+1}P_n^{\alpha,\beta}(x)\rbrack\\
        &=D_x^{\alpha+2}\lbrack(x-1)^{\alpha+1}P_n^{\alpha+1,\beta-1}(x)\rbrack
        -D_x^{\alpha+2}\lbrack(x-1)^{\alpha+1}P_{n-1}^{\alpha+1,\beta}(x)\rbrack\\  
        &=D_x\lbrack(n+1)_{\alpha+1}P_n^{0,\alpha+\beta}(x)\rbrack
                -D_x\lbrack(n)_{\alpha+1}P_{n-1}^{0,\alpha+\beta+1}(x)\rbrack\\    
        &=\frac{1}{2}(n+\alpha+\beta+1)\big\lbrack(n+1)_{\alpha+1}P_{n-1}^{1,\alpha+\beta+1}(x)-(n)_{\alpha+1}P_{n-2}^{1,\alpha+\beta+2}(x)\rbrack\\  
        &=\frac{1}{2}(n+\alpha+\beta+1)(n+1)_{\alpha}\big\lbrack (n+\alpha+1)P_{n-1}^{0,\alpha+\beta+2}(x)                               +(\alpha+1)P_{n-2}^{1,\alpha+\beta+2}(x)\big\rbrack. 
        \end{aligned}
        \end{equation*}
        Inserting this last identity into the right-hand side of (\ref{eq5.4}) and observing that, in view of (\ref{eq4.6}),
        \begin{equation*}
         \begin{aligned}
          D_x^{\alpha+2}\big\lbrack(x+1)^{\alpha+\beta+2}
          P_{n-1}^{0,\alpha+\beta+2}(x)\big\rbrack &=(n+\beta)_{\alpha+2}(x+1)^\beta
          P_{n-1}^{\alpha+2,\beta}(x),\\
         D_x^{\alpha+2}\big\lbrack(x+1)^{\alpha+\beta+2}
            P_{n-2}^{1,\alpha+\beta+2}(x)\big\rbrack &=(n+\beta-1)_{\alpha+2}(x+1)^\beta
            P_{n-2}^{\alpha+3,\beta}(x),
          \end{aligned}
          \end{equation*}
        we arrive at
        \begin{equation*}
        \begin{aligned}
         L_{2\alpha+4,x}^{\alpha,\beta}&P_n^{\alpha,\beta}(x)=
         (n)_{\alpha+2}(n+\beta)_{\alpha+2}\;\cdot\\
          &\cdot\big\lbrace \frac{n+\alpha+\beta+1}{n}\;\frac{x-1}{2} 
          P_{n-1}^{\alpha+2,\beta}(x)+\frac{(\alpha+1)(n+\beta-1)}{n(n+\alpha+1)}
         \;\frac{x-1}{2}P_{n-2}^{\alpha+3,\beta}(x)\big\rbrace.         
        \end{aligned}
        \end{equation*}
       Finally, due to (\ref{eq5.7}) and (\ref{eq5.6}), the two terms in curly brackets simplify to
       \begin{equation*}
       \begin{aligned}
       &\frac{n+\alpha+\beta+1}{n}\;\frac{x-1}{2} 
         P_{n-1}^{\alpha+2,\beta}(x)=P_n^{\alpha,\beta}(x)
         -\frac{\alpha+1}{n}P_{n-1}^{\alpha+1,\beta}(x),\\
       &\frac{(\alpha+1)(n+\beta-1)}{n(n+\alpha+1)}\;\frac{x-1}{2} 
       P_{n-2}^{\alpha+3,\beta}(x)=\frac{\alpha+1}{n}\big\lbrack P_{n-1}^{\alpha+1,\beta}(x)-\frac{\alpha+2}{n+\alpha+1}
       P_{n-1}^{\alpha+2,\beta}(x)\big\rbrack.
        \end{aligned}
        \end{equation*}   
        This proves the required identity (\ref{eq5.3}). 
        \end{proof} 
 \section{Some relations between the equations of Jacobi- and
  ultraspherical-type}
  \label{sec:6}
   
   We first observe a surprisingly simple connection between the higher-order differential expressions (\ref{eq1.13}) and (\ref{eq1.19}). 
   \begin{theorem}
   \label{thm6.1}
   The symmetric ultraspherical-type operator  $L_{2\alpha+4,x}$ and the Jacobi-type operator $L_{2\alpha+4,x}^{\alpha,\alpha+2}$ , both of the same order $2\alpha+4,\;\alpha \in \mathbb{N}_0$, are linked to each other via   
   \begin{equation}
    L_{2\alpha+4,x}[(x+1)y(x)]=(x+1)L_{2\alpha+4,x}^{\alpha,\alpha+2}y(x).
     \label{eq6.1}
    \end{equation}  
    \end{theorem}
  \begin{proof} 
 For any smooth function $\phi(x)$ and $m \in \mathbb{N}_0$, we have
 \begin{equation}
  \begin{aligned}
  D_x^m[(x+1)^{2m}D_x^m\phi(x)]&=\sum_{k=0}^{m}\binom{m}{k} 
  D_x^k[(x+1)^{2m}]D_x^{m-k}D_x^m\phi(x)\\
  &=(x+1)^m\sum_{k=0}^{m}\binom{m}{k}\binom{2m}{k}k!\;(x+1)^{m-k} D_x^{2m-k}\phi(x)\\ 
  &=(x+1)^m\sum_{k=0}^{m}\binom{2m}{k}D_x^k[(x+1)^m]D_x^{2m-k}\phi(x)\\ 
  &=(x+1)^m D_x^{2m}[(x+1)^m\phi(x)].      
  \end{aligned}
   \label{eq6.2}
   \end{equation}   
 So, starting with the Jacobi-type differential expression on the right-hand side of (\ref{eq6.1}) and choosing $m=\alpha+2$ and $\phi(x)=(x-1)^{\alpha+1}y(x)$ in identity (\ref{eq6.2}), we find that
 \begin{equation*}
   \begin{aligned}
   L_{2\alpha+4,x}^{\alpha,\alpha+2}y(x)&=
   (x-1)(x+1)^{-\alpha-2}
   D_x^{\alpha+2}\big\lbrack(x+1)^{2\alpha+4}D_x^{\alpha+2}\phi(x)\big\rbrack\\
   &=(x-1)D_x^{2\alpha+4}\big\lbrack(x+1)^{\alpha+2}(x-1)^{\alpha+1}y(x)
   \big\rbrack\\
   &=(x+1)^{-1}L_{2\alpha+4,x}[(x+1)y(x)].
   \end{aligned}
   \end{equation*}  
    \end{proof}
  \begin{remark}
  Identity (\ref{eq6.1}) follows as well by comparing the Jacobi-type factorization formula (\ref{eq4.1}) in case $\beta=\alpha+2$  with the factorized ultraspherical-type differential expression given in \cite[Sec.4-6]{Ma}. Indeed,
  \begin{equation*}
     \begin{aligned}
     (x+1)&L_{2\alpha+4,x}^{\alpha,\alpha+2}[(x-1)u(x)]=
     (x^2-1)\prod_{j=0}^{\alpha+1}\big\lbrack L_{2,x}^{\alpha+2,\alpha+2}+(j+1)(2\alpha+4-j)\big\rbrack u(x)\\
     &=(x^2-1)\prod_{j=0}^{\alpha+1}\big\lbrack (x^2-1)D_x^2+(2\alpha+6)x\;D_x+(j+1)(2\alpha+4-j)\big\rbrack u(x)\\
     &=(x^2-1)D_x^{2\alpha+4}\big\lbrack (x^2-1)^{\alpha+2}u(x)\big\rbrack=
     L_{2\alpha+4,x}[(x^2-1)u(x)].
      \end{aligned}
      \end{equation*}  
 \end{remark}   
   
   What the differential equations are concerned, a second, even more striking relationship is suggested by the two quadratic transformations due to Koornwinder \cite[(4.6-7)]{Ko},
   \begin{equation}
   P_{2n}^{\alpha,\alpha,N,N}(x)=p_nP_n^{\alpha,-1/2,0,2N}(2x^2-1),\;
   p_n=P_{2n}^{\alpha,\alpha,N,N}(1)/P_n^{\alpha,-1/2,0,2N}(1),
   \label{eq6.3}
   \end{equation}   
    \begin{equation}
    P_{2n+1}^{\alpha,\alpha,N,N}(x)=q_nx\;P_n^{\alpha,1/2,0,(4\alpha+6)N}(2x^2-1),\;q_n=P_{2n+1}^{\alpha,\alpha,N,N}(1)/P_n^{\alpha,1/2,0,(4\alpha+6)N}(1).
    \label{eq6.4}
    \end{equation}
  \\
    \begin{theorem}
    \label{thm6.2}
    Let $\alpha \in \mathbb{N}_0$, $N>0$. 
    \begin{itemize}
      \item[$(i)$]
   Considering  the differential equation (\ref{eq1.7}) associated with the Jacobi-type polynomials $y_n(x)=p_nP_n^{\alpha,-1/2,0,2N}(x)$, $n \in \mathbb{N}_0$, the substitution $x=2\xi^2-1$, $0 \le \xi \le 1$, leads to the equation (\ref{eq1.18}) for the even ultraspherical-type polynomials
    \begin{equation*}
    u_n(\xi):=y_n(2\xi^2-1)=P_{2n}^{\alpha,\alpha,N,N}(\xi).
     \end{equation*}
   \item[$(ii)$]  
   Under the same substitution as in part $(i)$ and a transformation of the dependent variable, the equation (\ref{eq1.7}) for the Jacobi-type polynomials 
   $y_n(x)=q_nP_n^{\alpha,1/2,0,(4\alpha+6)N}(x)$, $n \in \mathbb{N}_0$, reduces to the equation (\ref{eq1.18}) for the odd ultraspherical-type polynomials
    \begin{equation*}
        v_n(\xi)=\xi\;y_n(2\xi^2-1)=P_{2n+1}^{\alpha,\alpha,N,N}(\xi).
        \end{equation*}
     \end{itemize}
  \end{theorem}  
   \begin{proof} 
 (i) We multiply equation (\ref{eq1.7}) in the case that $y_n(x)=p_nP_n^{\alpha,-1/2,0,2N}(x)$, i.e.
  \begin{equation}
   2N\{L_{2\alpha+4,x}^{\alpha,-1/2}-\Lambda_{2\alpha+4,n}^{\alpha,-1/2}\}y(x)
     +C_{\alpha,-1/2}\{L_{2,x}^{\alpha,-1/2}-\Lambda_{2,n}^{\alpha,-1/2}\}y_nx)=0,\;-1<x<1,
   \label{eq6.5}
   \end{equation}  
   by $2^{2\alpha+3}$  and observe that the constants in (\ref{eq1.10}) and (\ref{eq1.17}) are related to each other by 
   \begin{equation*}
      2^{2\alpha+4}\Lambda_{2\alpha+4,n}^{\alpha,-1/2}=
      \Lambda_{2\alpha+4,2n},\; 2^{2\alpha+1}C_{\alpha,-1/2}=C_\alpha,\;
      4\Lambda_{2,n}^{\alpha,-1/2}=\Lambda_{2,2n}^{\alpha,\alpha}.
     \end{equation*}
   Concerning the two differential expressions in (\ref{eq6.5}) we formally replace $D_x$ by $(4\xi)^{-1}D_\xi$ in view of the substitution $x=2\xi^2-1$.  Years ago, we used already the so-called Bessel derivate $\delta_\xi=\xi^{-1}D_\xi$ in order to define the Bessel-type functions via a confluent limit of the Laguerre-type polynomials, see \cite{EM1}. Now we can use the expansion formula for the iterated Bessel derivatives $\delta_\xi^{m+1},\;m \in \mathbb{N}_0$, as stated in \cite[(2.8)]{EM1}. A quite lengthy, but straightforward induction argument then shows that for all $j=0.\dots,m$ and any smooth function $\phi(\xi)$,
   	\begin{equation*}
   	\delta_\xi^j\lbrack \xi^{2m+1} \delta_\xi^{m+1}\phi(\xi) \rbrack =
   	\sum_{k=j}^{m} \frac{(-2)^{k-m}(2m-k-j)!}{(m-k)!\;(k-j)!} \xi^{k-j}\;D_\xi^{k+j+1} \phi(\xi). 
    \end{equation*} 
   	Hence, for $j=m$, we end up with the surprisingly simple identity needed in the following,
   	\begin{equation}
    \delta_\xi^m\lbrack \xi^{2m+1} \delta_\xi^{m+1}\phi(\xi) \rbrack =
   	 D_\xi^{2m+1}\phi(\xi),\;m \in \mathbb{N}_0.
   	  \label{eq6.6}
     \end{equation}
   	In fact, choosing $m=\alpha+1$  and taking into account that $\xi\; \delta_\xi^{\alpha+2}=D_\xi\delta_\xi^{\alpha+1}$ we have
   	\begin{equation*}
     \xi\;\delta_\xi^{\alpha+2}\lbrack \xi^{2\alpha+3}\delta_\xi^{\alpha+2}\phi(\xi) \rbrack =
      	D_\xi D_\xi^{2\alpha+3}\phi(\xi)=D_\xi^{2\alpha+4}\phi(\xi)    
      \end{equation*}	
   and thus
   \begin{equation*}
     \begin{aligned}
   2^{2\alpha+4} L_{2\alpha+4,x}^{\alpha,-1/2} y_n(x)&=2^{2\alpha+4}
   \frac{x-1}{(x+1)^{-1/2}}D_x^{\alpha+2}\big\lbrace(x+1)^{\alpha+3/2}
   D_x^{\alpha+2}\lbrack (x-1)^{\alpha+1}y_n(x)\rbrack\big\rbrace\\
   &=(\xi^2-1)\;\xi\;\delta_\xi^{\alpha+2}\big\lbrace \xi^{2\alpha+3}\delta_\xi^{\alpha+2}\lbrack (\xi^2-1)^{\alpha+1}y_n(2\xi^2-1)\rbrack\big\rbrace\\
   &=(\xi^2-1)D_\xi^{2\alpha+4}\lbrack(\xi^2-1)^{\alpha+1}u_n(\xi)\rbrack\\ 
   &=L_{2\alpha+4,\xi}u_n(\xi).
    \end{aligned}
    \end{equation*}
   Moreover, it is not hard to see that
   \begin{equation*}
   \begin{aligned}
  4 L_{2,x}^{\alpha,-1/2} y_n(x)&=4\big\lbrace (x^2-1)D_x^2+\lbrack\alpha+1/2+(\alpha+3/2)x\rbrack \,D_x\big\rbrace y_n(x)\\
  &=\big\lbrace (\xi^2-1)D_\xi^2+2(\alpha+1)\xi\,D_\xi\big\rbrace y_n(2\xi^2-1)\\
  &=  L_{2,\xi}^{\alpha,\alpha}u_n(\xi).
     \end{aligned}
     \end{equation*}
   Putting all parts together, we arrive at equation (\ref{eq1.18}) applied to $u_n(\xi)=P_{2n}^{\alpha,\alpha,N,N}(\xi)$.\\
   (ii) In case of the Jacobi-type polynomials $y_n(x)=q_nP_n^{\alpha,1/2,0,(4\alpha+6)N}(x)$, equation (\ref{eq1.7}) reads
   \begin{equation}
    (4\alpha+6)N\{L_{2\alpha+4,x}^{\alpha,1/2}-\Lambda_{2\alpha+4,n}^{\alpha,1/2}\}
    y_n(x)+C_{\alpha,1/2}\{L_{2,x}^{\alpha,1/2}-\Lambda_{2,n}^{\alpha,1/2}\}
    y_n(x)=0,\;-1<x<1.
    \label{eq6.7}
     \end{equation} 
   Here we multiply this equation by $2^{2\alpha+4}(4\alpha+6)^{-1}\xi$, $\xi=\sqrt{(x+1)/2}$, and use
    \begin{equation*}
    2^{2\alpha+4}\Lambda_{2\alpha+4,n}^{\alpha,1/2}=
    \Lambda_{2\alpha+4,2n+1},\; 2^{2\alpha+2}(4\alpha+6)^{-1}
    C_{\alpha,1/2}=C_\alpha,\;
    4\Lambda_{2,n}^{\alpha,1/2}=\Lambda_{2,2n+1}^{\alpha,\alpha}-2(\alpha+1).       \end{equation*}
     Moreover, we find that
     \begin{equation*}
     \begin{aligned}
      4\xi\;L_{2,x}^{\alpha,1/2} y_n(x)&=4\xi\;\big\lbrace (x^2-1)D_x^2+\lbrack\alpha-1/2+(\alpha+5/2)x\rbrack \;D_x\big\rbrace y_n(x)\\
      &=\big\lbrace 4L_{2,x}^{\alpha,-1/2}+2(\alpha+1)\rbrace\lbrack\xi\;y_n(x)\rbrack=
      \big\lbrace L_{2,\xi}^{\alpha,\alpha}+2(\alpha+1)\rbrace v_n(\xi)
      \end{aligned}
      \end{equation*}
     and therefore
     \begin{equation*}
     4\xi\;\big\lbrace L_{2,x}^{\alpha,1/2}-\Lambda_{2,n}^{\alpha,1/2}\big\rbrace y_n(x)=
     \lbrace L_{2,\xi}^{\alpha,\alpha}-\Lambda_{2,2n+1}^{\alpha,\alpha}\rbrace v_n(\xi).
    \end{equation*} 
    Finally, we use identity (\ref{eq6.6}) again. With $m=\alpha+2$ and $\psi(\xi):=D_\xi\phi(\xi)$ we obtain
    \begin{equation*}
    \delta_\xi^{\alpha+2}\big\lbrace \xi^{2\alpha+5}\delta_\xi^{\alpha+2} \lbrack \xi^{-1}\psi(\xi) \rbrack \big\rbrace =
    \delta_\xi^{\alpha+2}\big\lbrace \xi^{2\alpha+5}\delta_\xi^{\alpha+3} \phi(\xi) \big\rbrace =D_\xi^{2\alpha+5}\phi(\xi)=D_\xi^{2\alpha+4}\psi(\xi).    
    \end{equation*}
   Hence,
   \begin{equation*}
   \begin{aligned}
   2^{2\alpha+4}\xi\;L_{2\alpha+4,x}^{\alpha,1/2} y_n(x)&=2^{2\alpha+4}
   \frac{\xi(x-1)}{(x+1)^{1/2}}D_x^{\alpha+2}\big\lbrace(x+1)^{\alpha+5/2}
   D_x^{\alpha+2}\lbrack (x-1)^{\alpha+1}y_n(x)\rbrack\big\rbrace\\
   &=(\xi^2-1)\delta_\xi^{\alpha+2}\big\lbrace \xi^{2\alpha+5}\delta_\xi^{\alpha+2}\lbrack \xi^{-1} (\xi^2-1)^{\alpha+1}\xi\;y_n(2\xi^2-1)\rbrack\big\rbrace\\
   &=(\xi^2-1)D_\xi^{2\alpha+4}\lbrack(\xi^2-1)^{\alpha+1}v_n(\xi)\rbrack\\ 
   &=L_{2\alpha+4,\xi}v_n(\xi).
    \end{aligned}
    \end{equation*}
   Combining all parts then yields equation (\ref{eq1.18}) applied to $v_n(\xi)=P_{2n+1}^{\alpha,\alpha,N,N}(\xi)$. Notice that by symmetry, the resulting equation for the ultrapherical-type polynomials, both of even and odd degree, can easily be extended to the full range $-1 \le \xi \le 1$.       
   \end{proof}  
\subsection*{Acknowledgement}
This paper is dedicated to the memory of the late Professor Ernst G$\ddot{o}$rlich. The author is greatly indebted to Professor G$\ddot{o}$rlich for his steady encouragement and support as a teacher and Ph.D. supervisor of the author as well as for the long lasting cooperation and friendship.

\vskip0.5cm
\begin{footnotesize}
\noindent
C. Markett, Lehrstuhl A f\"ur Mathematik, RWTH Aachen,
52056 Aachen, Germany;
\sPP
E-mail: {\tt markett@matha.rwth-aachen.de}

\end{footnotesize}

\end{document}